\documentclass[smallcondensed]{svjour3}                     
\smartqed  
\usepackage{graphicx}
\usepackage{mathptmx}      
\usepackage{amssymb}
\usepackage{amsmath}
\usepackage{graphicx}
\usepackage{fullpage}
\usepackage{verbatim}
\usepackage{tikz}
\usepackage{float}

\journalname{Geometriae Dedicata}
\begin{document}

\title{Hyperbolic reflection groups associated to the quadratic forms $-3x_0^2 + x_1^2 + \ldots + x_n^2$}


\author{John Mcleod}


\institute{John Mcleod \at
  Department of Mathematical Sciences\\
  Durham University\\
  Science Laboratories\\
  Durham\\
  DH1 3LE\\
  Tel: +44-(0)191-334-3050\\
  Fax: +44-(0)191-334-3051\\
  \email{j.a.mcleod@durham.ac.uk}
}

\date{Received: date / Accepted: date}

\maketitle

\begin{abstract}
We determine the maximal hyperbolic reflection groups associated to the quadratic forms $-3x_0^2 + x_1^2 + \ldots + x_n^2$, $n \ge 2$, and present the Coxeter schemes of their fundamental polyhedra. These groups exist in dimensions up to 13, and a proof is given that in higher dimensions these quadratic forms are not reflective.
\end{abstract}

\section{Introduction}
\label{intro}
Discrete groups generated by reflections in hyperplanes in spaces of constant curvature were defined by Coxeter \cite{coxeter34}. He studied these groups in spherical and Euclidean spaces. Towards the end of the 1960s, Vinberg initiated a program which sought to classify hyperbolic reflection groups. A recent breakthrough in this area was achieved independently by Nikulin \cite{nikulin07} and Agol, Belolipetsky, Storm, and Whyte \cite{agolbelolipetskystormwhyte08}; they proved that there exist only finitely many conjugacy classes of the maximal arithmetic hyperbolic reflection groups in all dimensions. This important result relied on the work of many people, most notably the foundational papers of Vinberg.

In parallel to classification, there was interest in examples of hyperbolic reflection groups, particularly in higher dimensions. The most famous examples were constructed by Vinberg and Kaplinskaya \cite{vinbergkaplinskaja78} in dimensions $n = 18$, and 19, and then by Borcherds \cite{borcherds87} for $n = 21$. We refer to the book by Conway and Sloane \cite{conwaysloane99} for an extended discussion of the geometric and algebraic properties of these beautiful constructions.

The method of \cite{agolbelolipetskystormwhyte08} indicates that in order to construct higher dimensional examples of arithmetic reflection groups one has to look at the arithmetic lattices of small covolume. Arithmetic discrete subgroups of the minimal covolume were recently described by Belolipetsky \cite{belolipetsky04} (for even dimension $n \ge 4$), and Emery \cite{emery09} (for odd dimensions $n \ge 5$; cf. also Belolipetsky-Emery \cite{belolipetskyemery10}). The cases $n=2$, and 3 were known already: the former thanks to an influential paper by Siegel \cite{siegel45}; the latter thanks to Chinburg-Friedman \cite{chinburgfriedman00}, and Meyerhoff \cite{meyerhoff85}. 

In this paper we will concentrate our attention on non-cocompact arithmetic subgroups. This is the case which is represented by the examples of Borcherds, Kaplinskaya, and Vinberg mentioned above. By \cite{belolipetsky04} and \cite{belolipetskyemery10}, in most cases the minimal covolume non-cocompact arithmetic lattices are defined by quadratic forms $-x_0^2 + x_1^2 + \ldots + x_n^2$. This case was studied by Vinberg and Vinberg-Kaplinskaya. Vinberg also considered quadratic forms $-2x_0^2 + x_1^2 + \ldots + x_n^2$, which naturally follow on. A somewhat unexpected result of \cite{belolipetskyemery10} shows that for some $n$, however, the minimal covolume is associated with the quadratic forms $-3x_0^2 + x_1^2 + \ldots + x_n^2$. More precisely, this happens when $n = 2r - 1$, $r \ge 4$ and even, i.e. $n = 7, 11, 15, \ldots$. Therefore, in this paper, we will focus our attention on the quadratic forms 
\[f(x) = -3x_0^2 + x_1^2 + \ldots + x_n^2,\]
and associated arithmetic subgroups.

Our principle tool is an algorithm which was developed by Vinberg in order to find a maximal subgroup generated by reflections of a given arithmetic lattice \cite{vinberg72(2)}. We are going to apply the algorithm to the integral lattices of the form $O(f, \mathbb{Z})$. As a result we obtain several new higher dimensional examples which have remarkable combinatorial symmetry. We were also able to show that for $n \ge 14$ the lattices under consideration are not reflective, i.e. do not contain a subgroup of finite index which is generated by reflections. The following theorem gives a precise statement of the main results of the paper:

\begin{theorem}\label{theorem1}
The groups of integral automorphisms of the quadratic form $-3x_0^2 + x_1^2 + \ldots + x_n^2$ are reflective for $2\le n\le 13$ and non-reflective for $n \ge 14$. The Coxeter diagrams of the fundamental polyhedra of the corresponding maximal reflection subgroups for $n=2$ to $13$ are given on Figures~\ref{one_to_eight} and \ref{ge_nine}.
\end{theorem}

The paper is organised as follows. In Section~\ref{background} we describe Vinberg's algorithm, along with two representations of the fundamental polyhedron of a reflective lattice: the Gram matrix and the Coxeter scheme. In Section~\ref{finitevolume} we give a necessary and sufficient condition for a hyperbolic polyhedron to have finite volume (or be bounded). In Section~\ref{results} we apply the two previous sections to the quadratic form $-3x_0^2 + x_1^2 + \ldots + x_n^2$ and obtain examples of discrete reflection groups in dimensions $n = 2$ to 13, thereby proving the first part of Theorem~\ref{theorem1}. We also include two corollaries in which we obtain a geometric description of some of the fundamental polyhedra. Finally, in Section~\ref{noexamplesinhigherdimensions} we complete the proof by showing that there are no higher dimensional reflective lattices associated to this quadratic form.

\section{Background}\label{background}
Let $\{v_0, v_1, \ldots, v_n\}$ be a basis of an $(n+1)$-dimensional vector space $E^{(n,1)}$ with the scalar multiplication of signature $(n, 1)$, given by the quadratic form 
\begin{equation} \label{quad_form}
-\varphi x_0^2 + x_1^2 + \ldots + x_n^2,
\end{equation}
where $\varphi$ is a positive integer. Consider
\[\{v \in E^{(n,1)} | (v, v) < 0\} = \mathfrak{C} \cup (-\mathfrak{C}),\]
where $\mathfrak{C}$ is an open convex cone. The \emph{vector model} of hyperbolic space $\mathbb{H}^n$ is the set of rays through the origin in $\mathfrak{C}$, or $\mathfrak{C}/\mathbb{R}^+$, such that the motions of $\mathbb{H}^n$ are the orthogonal transformations of $E^{(n,1)}$. The group of integer-valued automorphisms of (\ref{quad_form}) is then a subgroup of $Isom \, \mathbb{H}^n$, and will be denoted $\Theta_n$. It can be shown that it is discrete and that it has finite covolume. In accordance with \cite{vinberg72(2)}, denote by $\Gamma$ the normal subgroup of $\Theta_n$ generated by reflections, and $P$ a $\Gamma$-cell. Then
\begin{equation}\label{sdproduct}
\Theta_n = \Gamma \rtimes H
\end{equation}
where $H$ is a subgroup of the symmetry group of $P$.

In the vector model of $\mathbb{H}^n$, a hyperplane is given by the set of rays in $\mathfrak{C}$ which are orthogonal to a vector $e$ of positive length in $E^{(n,1)}$, and contained in a hyperbolic subspace of $E^{(n,1)}$. A hyperplane $\Pi_e$ defines two \emph{halfspaces}, $\Pi_e^+$ and $\Pi_e^-$, and a reflection which will be denoted $R_e$. For brevity, a hyperplane associated to a vector $e_i$ will be denoted $\Pi_i$. For $R_e \in \Gamma$, $e$ must have rational coefficients. Furthermore, the vector $e$ may be scaled such that all the coefficients are integral and do not have any common divisors. 

Let $e= \sum_{i=0}^nk_iv_i$. Then the action of $R_e$ on the basis vectors can be written as:
\begin{equation}\label{reflects}
  R_e v_j = \left\{
  \begin{array}{lr}
    v_j - \frac{2 k_j}{(e, e)}e, &  j > 0,\\
    v_j + \frac{2 \varphi k_j}{(e, e)}e, &  j = 0.
   \end{array}\right.
\end{equation}

There is a further condition for $R_e$ to be an element of $\Gamma$, namely the so-called \emph{Crystallographic condition}: Any pair of reflections $\alpha, \beta \in \Gamma$ must satisfy
\begin{equation}\label{crystcondition}
\frac{2(\alpha, \beta)}{(\beta, \beta)} \in \mathbb{Z}, 
\end{equation}
with respect to the quadratic form.

By linearity, we only need to check that $R_e$ satisfy this condition when applied to the basis vectors. Therefore it is necessary that both $\frac{2 k_j}{(e, e)}$ and $\frac{2 \varphi k_0}{(e, e)}$ are integral.

Vinberg's algorithm constructs the fundamental polyhedron of the maximal hyperbolic reflection subgroup of the integral automorphism group of a quadratic form. It begins by considering the stabiliser subgroup $\Gamma_0 \subset \Gamma$ of a point $x_0 \in \mathbb{H}^n$. The \emph{polyhedral angle} at $x_0$ is defined by
\[P_0 = \bigcap_{i=1}^k \Pi_i^-,\]
with all the hyperplanes being \emph{essential} (not wholly contained within another hyperplane). There is a unique $\Gamma$-cell which sits inside $P_0$ and contains $x_0$, and it shall be denoted $P$.

The algorithm continues by constructing further $\Pi_i$ such that 
\[P = \bigcap_{i} \Pi_i^-,\]
with the $Pi_i$s being essential, ordered by increasing $\rho (x_o, Pi_i)$ (where $\rho$ denotes hyperbolic distance), and $\Pi_i^-$ denoting the halfspace which contains $x_0$. If the basis vector $v_0$ is chosen such that it lies on the ray containing $x_0$, then the hyperbolic distance between $x_0$ and the hyperplane $\Pi_e$ is given by
\begin{equation}\label{rhox0pie}
\sinh^2\rho(x_0,\Pi_e) = -\frac{(e,v_0)^2}{(e,e)(v_0,v_0)}.
\end{equation}

When constructing the hyperplanes $\Pi_i$ for $i \ge k+1$, they must be chosen such that $\Pi_i$ is the closest mirror of $\Gamma$ to $x_0$ whose halfspace $\Pi_i^-$ contains an inner point of the intersection of all previously constructed halfspaces (this is equivalent to the normal vector $e_i$ having negative inner product with all previous normal vectors, with respect to the form (\ref{quad_form})). From (\ref{rhox0pie}), it is clear that finding the closest mirror is equivalent to minimising
\begin{equation}\label{minimise}
\frac{(e,v_0)^2}{(e,e)} = \frac{k_0^2}{(e,e)}.
\end{equation}

This process may be of finite or infinite length, and the termination criterion is given in Section~\ref{finitevolume}.

If the algorithm terminates then $P$ is an acute-angled polyhedron, and in fact it is a \emph{Coxeter polyhedron}. Recall that an acute-angled polyhedron is called a Coxeter polyhedron if all the dihedral angles at the intersections of pairs of faces are integer submultiples of $\pi$ (or zero). A complete presentation of an acute-angled polyhedron is given by a \emph{Gram matrix}. A Gram matrix $G = (g_{ij})$ is a symmetric matrix with entries:
\begin{equation*}
  g_{ij} = \left\{
  \begin{array}{ll}
    1 &  \; \text{ if } \; i = j,\\
    -\cos(\frac{\pi}{n}) &  \; \text{ if } \; \angle (\Pi_i, \Pi_j) = \frac{\pi}{n},\\
    -1 &  \; \text{ if } \; \angle (\Pi_i, \Pi_j) = 0,\\
    -\cosh(\rho(\Pi_i, \Pi_j)) &  \; \text{ if $\Pi_i$ and $\Pi_j$ do not intersect,}
   \end{array}\right.
\end{equation*}
where $\rho(\Pi_i, \Pi_j)$ is the minimum hyperbolic distance between the two hyperplanes. The entries of the Gram matrix may be computed directly from the normal vectors to the hyperplanes $\Pi_i$ as
\[g_{ij}=\frac{(e_i, e_j)}{(e_i, e_i)(e_j,e_j)}.\]

Another presentation of an acute-angled polyhedron which will be used is the Coxeter scheme. It is a graph which reproduces most of the information in the Gram matrix, with the exception of the distance between non-intersecting planes. Each vertex of a Coxeter scheme corresponds to a hyperplane, and the edges are as presented in Table~\ref{coxeteredges}.
\begin{table}[H]
\caption{\label{coxeteredges} The edges of a Coxeter diagram}
\centering
\begin{tabular}{l|l} 
  Type of edge & Corresponds to\\
  \hline
  comprised of $m-2$ lines, or labelled $m$ & a dihedral angle $\frac{\pi}{m}$\\
  a single heavy line & a ``cusp'', or a dihedral angle zero\\
  a dashed line & two divergent faces\\
  no line & a dihedral angle $\frac{\pi}{2}$
\end{tabular}
\end{table}

\begin{example}\label{simplexexample}
An example of a Coxeter scheme would be:
\begin{figure}[h!]
  \centering
  \begin{tikzpicture}
    \draw[thick] (0,0) -- (1, 0) node[above, xshift = -0.5cm]  {6};
    \draw[thick] (1,0) -- (2, 0);
    \draw[thick, double distance = 2pt] (2, 0) -- (3, 0);
    \filldraw[fill=white] (0, 0) circle (2pt); 
    \filldraw[fill=white] (1, 0) circle (2pt); 
    \filldraw[fill=white] (2, 0) circle (2pt); 
    \filldraw[fill=white] (3, 0) circle (2pt); 
  \end{tikzpicture}
\end{figure}

This corresponds to a noncompact simplex in 3 dimensional hyperbolic space with dihedral angles $\frac{\pi}{6}$, $\frac{\pi}{3}$, and $\frac{\pi}{4}$. The Gram matrix of a simplex can be recovered from its Coxeter scheme. In our case, we get
\[\begin{pmatrix} 1 & -\frac{1}{2} & 0 & -\frac{\sqrt{3}}{2}\\
 -\frac{1}{2} & 1 & -\frac{1}{\sqrt{2}} &  0\\
  0  & -\frac{1}{\sqrt{2}} & 1 & 0\\
 -\frac{\sqrt{3}}{2} & 0  & 0 &  1
    \end{pmatrix}.\]
\end{example}

\section{Determining whether an acute-angled polyhedron is of finite volume}\label{finitevolume}
In order to demonstrate that a Coxeter polyhedron is of finite volume, we appeal to a proposition of Vinberg \cite{vinberg72(2)}. The proposition is given in terms of properties of the Gram matrix, and it is reproduced here as Proposition~\ref{vin_prop_1}. The result can be reformulated in terms of the Coxeter scheme, and this is given below as Corollary~\ref{vin_unwritten_prop}.

The \em direct sum \em of matrices $A_1, \ldots, A_n$ is given by
\[\begin{pmatrix} A_1 &  &  & 0\\
    & A_2 &  &  \\
    &     & \ddots & \\
 0  &  &  &  A_n
    \end{pmatrix},\]
up to a permutation of the rows, and the same permutation of the columns. If a matrix $A$ cannot be presented as a direct sum of two non-empty matrices it is said to be \emph{irreducible}. Every symmetric matrix can be uniquely expressed as a direct sum of irreducible matrices, and these are known as its \emph{components}. The components of a matrix can be collected into three groups, denoted:
\begin{table}[h]
\centering
\begin{tabular}{l|l}
  $A^+$ & direct sum of all positive definite components of $A$,\\
  $A^0$ & direct sum of all degenerate non-negative definite components of $A$,\\
  $A^-$ & direct sum of all negative definite components of $A$.\\
\end{tabular}
\end{table}

If a Gram matrix is not positive definite, but all proper principal submatrices are, then the matrix is called \em critical. \em 
\begin{proposition}
A critical matrix is irreducible.
\end{proposition}
\begin{proof}
Assume that a critical matrix $M$ could be written as a direct sum of matrices $A_1 \oplus A_2 \oplus \ldots \oplus A_n$. The determinant of $M$ is then given by the product of the determinants of the $A_i$s. All proper, principal submatrices of $M$ are positive definite, so all the determinants of the $A_i$s will be positive, giving the determinant of $M$ as positive. However, $M$ is critical, so it must have determinant zero. Hence $M$ must be irreducible.
\end{proof}

Finally, let $P \subset \mathbb{H}^n$ be a non-degenerate, finite, acute-angled, convex polyhedron, with each face defined by a vector $e_i$, $i \in I$, and let $G$ be its Gram matrix. Let $S$ be a subset of the index set $I$. Define
\[K = \{v\in V |  (e_i, v) \le 0, \; \forall i \in I\}.\]
Then, for any set $S\subset I$, let
\[K_S=\{v \in K | (e_i, v) = 0 \text{ for } i \in S\}.\]

In the same way, define $G_S$ to be the submatrix of a matrix G defined by taking the rows and columns of $G$ indexed by the elements of $S$.

\begin{proposition}[Vinberg \cite{vinberg72(2)}]\label{vin_prop_1}
  The necessary and sufficient condition for the polyhedron $P$ to have finite volume is that, for any critical principal submatrix $G_S$ of the matrix $G$, either
\begin{enumerate}
\item if $G_S = G_S^0$, then there exists a $T \supset S$ such that $G_T = G_T^0$ and $rank \; G_T = n-1$, or
\item if $G_S=G_S^-$, then $K_S = \{0\}$
\end{enumerate}
\end{proposition}

Vinberg also provides an alternative approach which can simplify the verification of the second part of Proposition~\ref{vin_prop_1}:
\begin{proposition}[Vinberg \cite{vinberg72(2)}]\label{vin_prop_2}
Let the Gram matrix $G$ of the polyhedron $P$ be irreducible. If $S$ and $T \subset I$ are such that
\[G_{S\cup T} = G_S \oplus G_T, \; G_T = G_T^+,\]
then
\[K_{S\cup T} = \{0\}\implies K_S = \{0\}.\]
\end{proposition}

These results can be reformulated in terms of the Coxeter diagrams. Then Proposition~\ref{vin_prop_1} becomes
\begin{corollary}\label{vin_unwritten_prop}
A polyhedron $P$ has finite volume if, for any subgraph $G_S$ of the diagram, either
\begin{enumerate}
\item if $G_S$ is parabolic, then it is a connected component of a parabolic subgraph $G_T$ of the diagram which has rank $n-1$,
\item if $G_S$ is a broken-line branch or Lann\'{e}r subgraph, then by removing vertices the diagram can be disconnected into $G_S$ and an elliptic subgraph $G_T$ such that \[rank \; G_S + rank \; G_T = n+1.\] This latter condition is sufficient but not necessary for the polyhedron $P$ to have finite volume.
\end{enumerate}
\end{corollary}

There is an equivalent condition for \emph{bounded} polyhedra, for which the quotient space is compact (cf. Bugaenko \cite{bugaenko92}, Proposition~$1.1$).

\begin{example}
Returning to the simplex of Example~\ref{simplexexample}, we may now show that it has finite volume. The Coxeter scheme has a single critical subscheme, $\tilde{G_2}$, whose vertices are coloured white, and this scheme is parabolic:
\begin{figure}[h!]
  \centering
  \begin{tikzpicture}
    \draw[thick] (0,0) -- (1, 0) node[above, xshift = -0.5cm]  {6};
    \draw[thick] (1,0) -- (2, 0);
    \draw[thick, double distance = 2pt] (2, 0) -- (3, 0);
    \filldraw[fill=white] (0, 0) circle (2pt); 
    \filldraw[fill=white] (1, 0) circle (2pt); 
    \filldraw[fill=white] (2, 0) circle (2pt); 
    \filldraw (3, 0) circle (2pt); 
  \end{tikzpicture}
\end{figure}

$\tilde{G_2}$ has rank $2$, and as $n=3$, the simplex has finite volume. The simplex is not bounded as $\tilde{G_2}$ is a parabolic subscheme.
\end{example}

\section{Results}\label{results}
In this section we use Vinberg's algorithm to construct maximal reflection subgroups of the groups of integral automorphisms of the quadratic forms
\begin{equation}
f(x) = -3x_0^2 + x_1^2 + \ldots + x_n^2.
\end{equation}

Following Vinberg, we fix a point $x_0 \in \mathbb{H}^n$, and the vectors which will be chosen as the polyhedral angle at $x_0$ are:
\begin{align*}
e_i &= -v_i + v_{i+1} \; \mathrm{for}\; 1 \le i < n,\\
e_n &= -v_n.
\end{align*}
Each new vector must have negative inner product with all previous vectors with respect to the form $f$. Therefore upon each new hyperplane corresponding to the normal vector $e= \sum_{i=0}^nk_iv_i$, there is the following ordering condition on the coefficients $k_i$, $i>0$:
\begin{equation}
k_1\ge k_2 \ge \ldots \ge k_n \ge 0. \label{k_i_ordering}
\end{equation}
The halfspace associated to each new hyperplane is chosen to be the halfspace which contains $x_0$. Therefore each new hyperplane corresponding to the normal vector $e$ must satisfy:
\[(e, v_0) < 0,\]
where the bilinear form $(,)$ is the inner product defined by $f$. This statement implies that
\begin{equation}
k_0 > 0. \label{positive_k_0}
\end{equation}
By the crystallographic condition (\ref{crystcondition}), we can determine the possible values of $(e,e)$. We obtain that $(e, e)$ could equal $3$ or $6$, as long as all the $k_j$ are divisible by $3$, therefore the possible values are:
\begin{equation}\label{epsilon}
(e, e) = 1, 2, 3, 6.
\end{equation}

We now prove the first part of Theorem~\ref{theorem1}:
\begin{proposition}\label{vector_prop}
Given the preceding conditions, the sets of vectors which are found by the algorithm are presented in Table~\ref{results_table}.
\end{proposition}
\begin{proof}
The algorithm searches for vectors $(k_0, k_1, \ldots, k_n)$ which satisfy the relations (\ref{k_i_ordering}), (\ref{positive_k_0}), and (\ref{epsilon}). The vector must have non-positive inner product with all vectors which have been found before it. Finally, if the length is divisible by $3$ then all the $k_i$, $i > 0$, must be also divisible by $3$. Of all the vectors which satisfy these conditions, the vector which minimises the quantity (\ref{minimise}) is chosen. This way we obtain the following vectors:

\begin{enumerate}
\item $v_0 + 3 v_1$

The vector which minimises (\ref{minimise}) should have length $6$ and $k_0 = 1$, so it remains to show that such a vector would satisfy the above constraints. By (\ref{epsilon}), if $(e, e) = 6$, all $k_i$s, $i > 0$, must be divisible by $3$. Under these conditions, a solution is sought for the equation
\[(e, e) + 3 k_0^2 = 9 = \sum_{i=1}^n k_i^2.\]
It is clear that this is solved by $k_i = 3$, $k_j = 0$, $j \neq i$, and by (\ref{k_i_ordering}), $i = 1$.

As all subsequent vectors must have negative inner product with this vector, another constraint is imposed:
\begin{equation}\label{k_0_ordering}
k_0 \ge k_1.
\end{equation}

For $n=3$ the algorithm terminates here, as this vector is sufficient to define an acute-angled polyhedron of finite volume.

\item $v_0 + v_1 + v_2 + v_3 + v_4$ and $v_0 + v_1 + v_2 + v_3 + v_4 + v_5$

After $\frac{1}{6}$, the next possible values of (\ref{minimise}) are as follows:
\begin{enumerate}
\item $\frac{1}{3}$: $k_0 = 1$, $(e, e) = 3$, 
\item $\frac{1}{2}$: $k_0 = 1$, $(e, e) = 2$, 
\item $\frac{2}{3}$: $k_0 = 2$, $(e, e) = 6$, 
\item $1$: $k_0 = 1$, $(e, e) = 1$.
\end{enumerate}

By the crystallographic condition, and (\ref{k_0_ordering}), the cases $\frac{1}{3}$ and $\frac{2}{3}$ are not possible. The second case, $\frac{1}{2}$, is realised by a solution to the Diophantine equation
\[(e, e) + 3 k_0^2 = 5 = \sum_{i=1}^n k_i^2,\]
where, by (\ref{k_0_ordering}) and (\ref{k_i_ordering}), all the $k_i$ must be bounded above by 1. Therefore this equation only has solutions in 5 or more dimensions, and produces
\begin{equation}\label{5dimensions}
v_0 + v_1 + v_2 + v_3 + v_4 + v_5.
\end{equation}

Now consider the final case in this list. This is realised by a solution to the Diophantine equation
\[(e, e) + 3 k_0^2 = 4 = \sum_{i=1}^n k_i^2,\]
where again, by (\ref{k_0_ordering}) and (\ref{k_i_ordering}), all the $k_i$ must be bounded above by 1. Therefore this equation has solutions in 4 or more dimensions, and produces
\begin{equation}\label{4dimensions}
v_0 + v_1 + v_2 + v_3 + v_4.
\end{equation}

A new vector is required in 4 dimensions to define an acute angled polyhedron of finite volume, and the vector (\ref{4dimensions}) is sufficient. In 5 or more dimensions the new vector must be taken to be (\ref{5dimensions}), as it has a smaller value of (\ref{minimise}). Note that as the inner product of (\ref{5dimensions}) and (\ref{4dimensions}) is positive, the two vectors are not mutually admissable.

In 5 or more dimensions, the additional constraint coming from $v_0 + v_1 + v_2 + v_3 + v_4 + v_5$ is:
\begin{equation} \label{3k_0gesum}
3 k_0 \ge k_1 + k_2 + k_3 + k_4 + k_5.
\end{equation}

\item $2 v_0 + v_1 + \ldots + v_{13}$ and $2( v_0 + v_1) + v_2 + \ldots + v_{11}$

After 1, the next possible values of (\ref{minimise}) are as follows:
\begin{enumerate}
\item $\frac{4}{3}$: $k_0 = 2$, $(e, e) = 3$, 
\item $\frac{3}{2}$: $k_0 = 3$, $(e, e) = 6$, 
\item 2: $k_0 = 2$, $(e, e) = 2$.
\end{enumerate}

Again, by the crystallographic condition, and (\ref{k_0_ordering}), the case $\frac{4}{3}$ is not possible. While the second case, $\frac{3}{2}$, is permitted by these two conditions, it requires a solution to the Diophantine equation
\[(e, e) + 3 k_0^2 = 33 = \sum_{i=1}^n k_i^2 = 9\sum_{i=1}^n {k_i^\prime}^2,\]
where $k_i = 3k_i^\prime$, and $9 \nmid 33$, so there are no solutions of this form.

Therefore consider the final case. This requires a solution to the Diophantine equation
\[(e, e) + 3 k_0^2 = 14 = \sum_{i=1}^n k_i^2.\]
There are two partitions of 14 into sums of squares respecting both (\ref{k_0_ordering}) and (\ref{3k_0gesum}), and they are:
\begin{enumerate}
\item 2, 1, 1, 1, 1, 1, 1, 1, 1, 1, 1;
\item 1, 1, 1, 1, 1, 1, 1, 1, 1, 1, 1, 1, 1, 1.
\end{enumerate}
The first of these represents the vector $2( v_0 + v_1) + v_2 + \ldots + v_{11}$, and as such arises in 11 or more dimensions, while the second, $2 v_0 + v_1 + \ldots + v_{13}$, does not appear until $n=13$. The inner product between them is zero, so they are mutually admissable.

\item Remaining vectors

The remaining vectors in Table~(\ref{results_table}) arise in the same way, and we omit the details.
\end{enumerate}

\end{proof}

The Coxeter schemes corresponding to the hyperbolic reflection groups found by this algorithm are presented in Figure~\ref{one_to_eight} and Figure~\ref{ge_nine}. The diagrams have been split in this way to highlight the different approaches which must be employed to demonstrate that the polyhedra have finite volume.
\begin{table}[h!]
\centering
\caption{Vectors produced by Vinberg's algorithm.}
  \begin{tabular}{c|c|c|c|c} \label{results_table}
    $i$ & $e_i$ & $(e,e)$ & $n$ & $\frac{k_0^2}{(e,e)}$\\
    \hline
    $n+1$ & $v_0 + 3 v_1$ & $6$ & $\ge 1$ & $0.167$\\
    $n+2$ & $v_0 + v_1 + v_2 + v_3 + v_4$ & $1$ & $4$ & $1$\\
    & $v_0 + v_1 + v_2 + v_3 + v_4 + v_5$ & $2$ & $\ge 5$ & $0.5$\\
    $n+3$ & $5 v_0 + 3(v_1 + v_2 + \ldots + v_9)$ & $6$ & $\ge 9$ & $4.167$\\
    $n+4$ & $2( v_0 + v_1) + v_2 + \ldots + v_{10}$ & $1$ & $10$ & $4$\\
    & $2( v_0 + v_1) + v_2 + \ldots + v_{11}$ & $2$ & $\ge 11$ & $2$\\
    $n+5$ & $3( v_0 + v_1 + v_2) + v_3 + \ldots + v_{12}$ & $1$ & $12$ & $9$\\
    & $3( v_0 + v_1 + v_2) + v_3 + \ldots + v_{13}$ & $2$ & $\ge 13$ & $4.5$\\
    $n+6$ & $5 v_0 + 3(v_1 + v_2 + \ldots + v_8) + v_9 + v_{10} + v_{11} + v_{12}$ & $1$ & $12$ & $25$\\
    & $5 v_0 + 3(v_1 + v_2 + \ldots + v_8) + v_9 + v_{10} + v_{11} + v_{12} + v_{13}$ & $2$ & $\ge 13$ & $12.5$\\
    $n+7$ & $2 v_0 + v_1 + \ldots + v_{13}$ & $1$ & $13$ & $4$\\
    $n+8$ & $8 v_0 + 6(v_1 + v_2 + v_3) + 3(v_4 + \ldots + v_{13})$ & $6$ & $\ge 13$ & $10.667$\\
    $n+9$ & $10 v_0 + 6(v_1 + \ldots + v_7) + 3(v_8 + \ldots + v_{13})$ & $6$ & $\ge 13$ & $16.667$\\
  \end{tabular}
\end{table}

The diagrams in Figure~\ref{one_to_eight} all have no broken-line branches or Lann\'{e}r subgraphs, and each parabolic subgraph is a connected component of a parabolic subgraph of rank $n-1$, so by Corollary~\ref{vin_unwritten_prop}, all have finite volume. This can be easily checked by inspection: removing the black vertex (where present) leaves a parabolic subscheme of rank $n-1$.

Note that in the case $n=2$ we get a Lann\'{e}r graph and hence a compact polyhedron, while for $n \ge 3$ the polyhedra are non-compact. This agrees with the well known compactness criterion which implies that for $n \ge 4$ a lattice defined by a quadratic form is non-cocompact if and only if the form is defined over $\mathbb{Q}$ (cf. \cite{limillson93}, section 1).

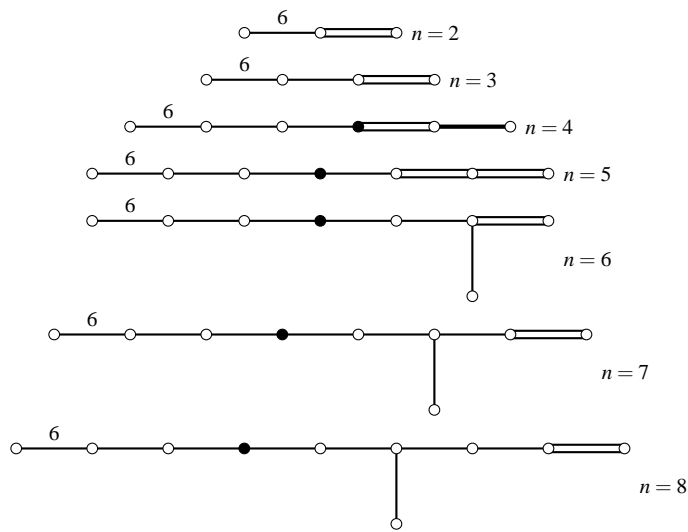
\begin{figure}[h!]
\centering
\caption{\label{one_to_eight} Coxeter diagrams of the fundamental polyhedra of the discrete reflection group corresponding to the automorphism groups of the quadratic form $-3x_0^2 + x_1^2 + \ldots + x_n^2$, for $n = 2$ to $8$.}
  \begin{tikzpicture}
    \draw[thick] (0,0) -- (1, 0) node[above, xshift = -0.5cm]  {6};
    \draw[thick, double distance = 2pt] (1, 0) -- (2, 0);
    \filldraw[fill=white] (0, 0) circle (2pt); 
    \filldraw[fill=white] (1, 0) circle (2pt); 
    \filldraw[fill=white] (2, 0) circle (2pt) node[right, xshift = 0.1cm] {$n=2$}; 
  \end{tikzpicture}

  \begin{tikzpicture}
    \draw[thick] (0,0) -- (1, 0) node[above, xshift = -0.5cm]  {6};
    \draw[thick] (1,0) -- (2, 0);
    \draw[thick, double distance = 2pt] (2, 0) -- (3, 0);
    \filldraw[fill=white] (0, 0) circle (2pt); 
    \filldraw[fill=white] (1, 0) circle (2pt); 
    \filldraw[fill=white] (2, 0) circle (2pt); 
    \filldraw[fill=white] (3, 0) circle (2pt) node[right, xshift = 0.1cm] {$n=3$}; 
  \end{tikzpicture}

  \begin{tikzpicture}
    \draw[thick] (0,0) -- (1, 0) node[above, xshift = -0.5cm]  {6};
    \draw[thick] (1,0) -- (2, 0);
    \draw[thick] (2,0) -- (3, 0);
    \draw[thick, double distance = 2pt] (3, 0) -- (4, 0);
    \draw[ultra thick] (4,0) -- (5, 0);    
    \filldraw[fill=white] (0, 0) circle (2pt); 
    \filldraw[fill=white] (1, 0) circle (2pt); 
    \filldraw[fill=white] (2, 0) circle (2pt); 
    \filldraw (3, 0) circle (2pt); 
    \filldraw[fill=white] (4, 0) circle (2pt); 
    \filldraw[fill=white] (5, 0) circle (2pt) node[right, xshift = 0.1cm] {$n=4$}; 
  \end{tikzpicture}

  \begin{tikzpicture}
    \draw[thick] (0,0) -- (1, 0) node[above, xshift = -0.5cm]  {6};
    \draw[thick] (1,0) -- (2, 0);
    \draw[thick] (2,0) -- (3, 0);
    \draw[thick] (3,0) -- (4, 0);
    \draw[thick, double distance = 2pt] (4, 0) -- (5, 0);
    \draw[thick, double distance = 2pt] (5, 0) -- (6, 0);
    \filldraw[fill=white] (0, 0) circle (2pt); 
    \filldraw[fill=white] (1, 0) circle (2pt); 
    \filldraw[fill=white] (2, 0) circle (2pt); 
    \filldraw (3, 0) circle (2pt); 
    \filldraw[fill=white] (4, 0) circle (2pt); 
    \filldraw[fill=white] (5, 0) circle (2pt); 
    \filldraw[fill=white] (6, 0) circle (2pt) node[right, xshift = 0.1cm] {$n=5$}; 
  \end{tikzpicture}

  \begin{tikzpicture}
    \draw[thick] (0,0) -- (1, 0) node[above, xshift = -0.5cm]  {6};
    \draw[thick] (1,0) -- (2, 0);
    \draw[thick] (2,0) -- (3, 0);
    \draw[thick] (3,0) -- (4, 0);
    \draw[thick] (4, 0) -- (5, 0);
    \draw[thick, double distance = 2pt] (5, 0) -- (6, 0);
    \draw[thick] (5, 0) -- (5, -1);
    \filldraw[fill=white] (0, 0) circle (2pt); 
    \filldraw[fill=white] (1, 0) circle (2pt); 
    \filldraw[fill=white] (2, 0) circle (2pt); 
    \filldraw (3, 0) circle (2pt); 
    \filldraw[fill=white] (4, 0) circle (2pt); 
    \filldraw[fill=white] (5, 0) circle (2pt); 
    \filldraw[fill=white] (5, -1) circle (2pt); 
    \filldraw[fill=white] (6, 0) circle (2pt) node[right, xshift = 0.1cm, yshift = -0.5cm] {$n=6$}; 
  \end{tikzpicture}

  \begin{tikzpicture}
    \draw[thick] (0,0) -- (1, 0) node[above, xshift = -0.5cm]  {6};
    \draw[thick] (1,0) -- (2, 0);
    \draw[thick] (2,0) -- (3, 0);
    \draw[thick] (3,0) -- (4, 0);
    \draw[thick] (4, 0) -- (5, 0);
    \draw[thick] (5, 0) -- (6, 0);
    \draw[thick, double distance = 2pt] (6, 0) -- (7, 0);
    \draw[thick] (5, 0) -- (5, -1);
    \filldraw[fill=white] (0, 0) circle (2pt); 
    \filldraw[fill=white] (1, 0) circle (2pt); 
    \filldraw[fill=white] (2, 0) circle (2pt); 
    \filldraw (3, 0) circle (2pt); 
    \filldraw[fill=white] (4, 0) circle (2pt); 
    \filldraw[fill=white] (5, 0) circle (2pt); 
    \filldraw[fill=white] (5, -1) circle (2pt); 
    \filldraw[fill=white] (6, 0) circle (2pt); 
    \filldraw[fill=white] (7, 0) circle (2pt) node[right, xshift = 0.1cm, yshift = -0.5cm] {$n=7$}; 
  \end{tikzpicture}

  \begin{tikzpicture}
    \draw[thick] (0,0) -- (1, 0) node[above, xshift = -0.5cm]  {6};
    \draw[thick] (1,0) -- (2, 0);
    \draw[thick] (2,0) -- (3, 0);
    \draw[thick] (3,0) -- (4, 0);
    \draw[thick] (4, 0) -- (5, 0);
    \draw[thick] (5, 0) -- (6, 0);
    \draw[thick] (6, 0) -- (7, 0);
    \draw[thick, double distance = 2pt] (7, 0) -- (8, 0);
    \draw[thick] (5, 0) -- (5, -1);
    \filldraw[fill=white] (0, 0) circle (2pt); 
    \filldraw[fill=white] (1, 0) circle (2pt); 
    \filldraw[fill=white] (2, 0) circle (2pt); 
    \filldraw (3, 0) circle (2pt); 
    \filldraw[fill=white] (4, 0) circle (2pt); 
    \filldraw[fill=white] (5, 0) circle (2pt); 
    \filldraw[fill=white] (5, -1) circle (2pt); 
    \filldraw[fill=white] (6, 0) circle (2pt); 
    \filldraw[fill=white] (7, 0) circle (2pt); 
    \filldraw[fill=white] (8, 0) circle (2pt) node[right, xshift = 0.1cm, yshift = -0.5cm] {$n=8$}; 
  \end{tikzpicture}

\end{figure}

The diagrams in Figure~\ref{ge_nine} do include examples of broken-line branches, and Lann\'{e}r subgraphs. In each case these may be addressed using the sufficient condition in the second part of Corollary~\ref{vin_unwritten_prop}. However, the parabolic subgraphs still need to be considered, as for the previous diagrams, and they can be seen by inspection to be connected components of parabolic subgraphs of the appropriate rank.

Consider $n=9$. By deleting the two vertices which connect the broken-line branch to the rest of the diagram it can be seen that a copy of the elliptic graph $E_8$ remains. A broken-line branch has rank $2$, and $E_8$ has rank $8$, and therefore as $2+8 = 9 + 1 = n + 1$, the polyhedron has finite volume.

Now consider $n=10$. As the graph is symmetric only one of the copies of the Lann\'{e}r subgraph will be considered. Incidentally, this Lann\'{e}r graph has already appeared, as the simplex when $n=2$.  Again, by deleting vertices which connect the Lann\'{e}r subgraph to the rest of the diagram it can be seen that a copy of the elliptic graph $E_8$ remains. The Lann\'{e}r subgraph has rank $3$, and again $E_8$ has rank $8$, and therefore as $3+8 = 10 + 1 = n+1$, the polyhedron has finite volume.

The remaining graphs are dealt with in precisely the same way, and therefore the details will be omitted.

\begin{figure}[h!]
\centering
\caption{\label{ge_nine} Coxeter diagrams of the fundamental polyhedra of the discrete reflection group corresponding to the automorphism groups of the quadratic form $-3x_0^2 + x_1^2 + \ldots + x_n^2$ for $n = 9$ to $13$.}
  \begin{tikzpicture}
        \draw[thick, double distance = 2pt] (0, 0) -- (1, 0);
        \draw[thick, dashed] (1,0) -- (2, 0);
        \draw[ultra thick] (2,0) -- (3, 0);
        \draw[thick] (3,0) -- (4, 0) node[above, xshift = -0.5cm]  {6};
        \draw[thick] (0,0) -- (0, -1);
        \draw[thick] (4,0) -- (4, -1);
        \draw[thick] (0, -1) -- (0.8, -1);
        \draw[thick] (0.8, -1) -- (1.6, -1);
        \draw[thick] (1.6, -1) -- (2.4, -1);
        \draw[thick] (2.4, -1) -- (3.2, -1);
        \draw[thick] (3.2, -1) -- (4, -1);
        \draw[thick] (1.6, -1) -- (1.6, -0.2);
        \filldraw[fill=white] (0, 0) circle (2pt); 
        \filldraw[fill=white] (1, 0) circle (2pt); 
        \filldraw[fill=white] (2, 0) circle (2pt); 
        \filldraw[fill=white] (3, 0) circle (2pt); 
        \filldraw[fill=white] (4, 0) circle (2pt)node[right, xshift = 0.1cm, yshift = -0.5cm] {$n=9$}; 
        \filldraw[fill=white] (0, -1) circle (2pt); 
        \filldraw[fill=white] (0.8, -1) circle (2pt); 
        \filldraw[fill=white] (1.6, -1) circle (2pt); 
        \filldraw[fill=white] (2.4, -1) circle (2pt); 
        \filldraw[fill=white] (3.2, -1) circle (2pt); 
        \filldraw[fill=white] (4, -1) circle (2pt); 
        \filldraw[fill=white] (1.6, -0.2) circle (2pt); 
  \end{tikzpicture}

  \begin{tikzpicture}
        \draw[ultra thick] (0,0) -- (1, 0);
        \draw[thick, double distance = 2pt] (1,0) -- (1.8412535 , -0.5406408);
        \draw[thick, double distance = 2pt] (0,0) -- (-0.8412535 , -0.5406408);

        \draw[thick] (1.8412535 , -0.5406408) -- (1, -1.0812816)  node[above, xshift=0.32cm, yshift=0.2cm]  {6};
        \draw[thick] (-0.8412535 , -0.5406408) -- (0, -1.0812816) node[above, xshift=-0.32cm, yshift=0.2cm]  {6};;
        \draw[ultra thick] (0, -1.0812816) -- (1, -1.0812816);

        \draw[thick] (1.8412535 , -0.5406408) -- (2.2566685 , -1.4502728);
        \draw[thick] (-0.8412535 , -0.5406408) -- (-1.2566685 , -1.4502728);
        \draw[thick] (2.2566685 , -1.4502728) -- (2.1143537, -2.4400943);
        \draw[thick] (-1.2566685 , -1.4502728) -- (-1.1143537, -2.4400943);
        \draw[thick] (2.1143537, -2.4400943) -- (1.4594930, -3.1958438);
        \draw[thick] (-1.1143537, -2.4400943) -- (-0.4594930, -3.1958438);
        \draw[thick] (1.4594930, -3.1958438) -- (0.5000000, -3.4775764);
        \draw[thick] (-0.4594930, -3.1958438) -- (0.5000000, -3.4775764);
        \draw[thick] (0.5000000, -3.4775764) -- (0.5000000, -2.4775764);

        \filldraw[fill=white] (0, 0) circle (2pt); 
        \filldraw[fill=white] (1, 0) circle (2pt); 
        \filldraw[fill=white] (1.8412535 , -0.5406408) circle (2pt); 
        \filldraw[fill=white] (-0.8412535 , -0.5406408) circle (2pt); 
        \filldraw[fill=white] (2.2566685 , -1.4502728) circle (2pt); 
        \filldraw[fill=white] (-1.2566685 , -1.4502728) circle (2pt) node[left, xshift = -0.1cm, yshift = -0.5cm] {$n=10$}; 
        \filldraw[fill=white] (2.1143537, -2.4400943) circle (2pt); 
        \filldraw[fill=white] (-1.1143537, -2.4400943) circle (2pt); 
        \filldraw[fill=white] (1.4594930, -3.1958438) circle (2pt); 
        \filldraw[fill=white] (-0.4594930, -3.1958438) circle (2pt); 
        \filldraw[fill=white] (0.5000000, -3.4775764) circle (2pt); 
        \filldraw[fill=white] (0.5000000, -2.4775764) circle (2pt); 
        \filldraw[fill=white] (0, -1.0812816) circle (2pt); 
        \filldraw[fill=white] (1, -1.0812816) circle (2pt); 
  \end{tikzpicture}
  \begin{tikzpicture}
    \draw[thick, double distance = 2pt] (0,0) -- (0.9659258 , -0.2588190);
    \draw[thick, double distance = 2pt] (0,0) -- (-0.9659258 , -0.2588190);

    \draw[thick] (0.9659258 , -0.2588190) -- (1.6730326, -0.9659258);
    \draw[thick] (-0.9659258 , -0.2588190) -- (-1.6730326, -0.9659258);    

    \draw[thick] (1.6730326, -0.9659258) -- (0.6730326, -0.9659258)  node[below, xshift=0.5cm]  {6};
    \draw[thick] (-1.6730326, -0.9659258) -- (-0.6730326, -0.9659258)  node[below, xshift=-0.5cm]  {6};
    \draw[ultra thick] (-0.6730326, -0.9659258) -- (0.6730326, -0.9659258);

    \draw[thick] (1.6730326, -0.9659258) -- (1.9318517, -1.9318517);
    \draw[thick] (-1.6730326, -0.9659258) -- (-1.9318517, -1.9318517);    
    \draw[thick] (1.9318517, -1.9318517) -- (1.6730326, -2.8977775);
    \draw[thick] (-1.9318517, -1.9318517) -- (-1.6730326, -2.8977775);    
    \draw[thick] (1.6730326, -2.8977775) -- (0.9659258, -3.6048843);
    \draw[thick] (-1.6730326, -2.8977775) -- (-0.9659258, -3.6048843);    
    \draw[thick] (0.9659258, -3.6048843) -- (0, -3.8637033);
    \draw[thick] (-0.9659258, -3.6048843) -- (0, -3.8637033);    
    \draw[thick] (0, -3.8637033) -- (0, -2.8637033);    

    \filldraw[fill=white] (0 , 0) circle (2pt); 
    \filldraw[fill=white] (0.9659258 , -0.2588190) circle (2pt); 
    \filldraw[fill=white] (-0.9659258 , -0.2588190) circle (2pt); 
    \filldraw[fill=white] (1.6730326, -0.9659258) circle (2pt); 
    \filldraw[fill=white] (-1.6730326, -0.9659258) circle (2pt); 
    \filldraw[fill=white] (1.9318517, -1.9318517) circle (2pt) node[right, xshift = 0.1cm] {$n=11$}; 
    \filldraw[fill=white] (-1.9318517, -1.9318517) circle (2pt); 
    \filldraw[fill=white] (1.6730326, -2.8977775) circle (2pt); 
    \filldraw[fill=white] (-1.6730326, -2.8977775) circle (2pt); 
    \filldraw[fill=white] (0.9659258, -3.6048843) circle (2pt); 
    \filldraw[fill=white] (-0.9659258, -3.6048843) circle (2pt); 
    \filldraw[fill=white] (0, -3.8637033) circle (2pt); 
    \filldraw[fill=white] (0, -2.8637033) circle (2pt); 
    \filldraw[fill=white] (0.6730326, -0.9659258) circle (2pt); 
    \filldraw[fill=white] (-0.6730326, -0.9659258) circle (2pt); 
  \end{tikzpicture}

  \begin{tikzpicture}
    \draw[thick] (0,0) -- (0.9659258 , -0.2588190);
    \draw[thick] (0,0) -- (-0.9659258 , -0.2588190);
    \draw[thick] (0,0) -- (0, -1);

    \draw[thick] (0.9659258 , -0.2588190) -- (1.6730326, -0.9659258);
    \draw[thick] (-0.9659258 , -0.2588190) -- (-1.6730326, -0.9659258);    

    \draw[thick] (1.6730326, -0.9659258) -- (0.807007196, -1.4659258)  node[below, xshift = 0.5cm, yshift = 0.35cm]  {6};
    \draw[thick] (-1.6730326, -0.9659258) -- (-0.807007196, -1.4659258)  node[below, xshift = -0.5cm, yshift = 0.35cm]  {6};
    \draw[ultra thick] (0.807007196, -1.4659258) -- (-0.807007196, -1.4659258);

    \draw[dashed] (1.9318517, -1.9318517) -- (0.807007196, -2.3977775);
    \draw[dashed] (-1.9318517, -1.9318517) -- (-0.807007196, -2.3977775);
    \draw[dashed] (0.807007196, -2.3977775) -- (-0.807007196, -2.3977775);
    \draw[dashed] (0.807007196, -2.3977775) -- (-0.807007196, -1.4659258);
    \draw[dashed] (-0.807007196, -2.3977775) -- (0.807007196, -1.4659258);

    \draw[ultra thick] (0.807007196, -2.3977775) -- (0, -1);
    \draw[ultra thick] (-0.807007196, -2.3977775) -- (0, -1);

    \draw[thick] (1.6730326, -0.9659258) -- (1.9318517, -1.9318517);
    \draw[thick] (-1.6730326, -0.9659258) -- (-1.9318517, -1.9318517);    
    \draw[thick] (1.9318517, -1.9318517) -- (1.6730326, -2.8977775);
    \draw[thick] (-1.9318517, -1.9318517) -- (-1.6730326, -2.8977775);    
    \draw[thick] (1.6730326, -2.8977775) -- (0.9659258, -3.6048843);
    \draw[thick] (-1.6730326, -2.8977775) -- (-0.9659258, -3.6048843);    
    \draw[thick] (0.9659258, -3.6048843) -- (0, -3.8637033);
    \draw[thick] (-0.9659258, -3.6048843) -- (0, -3.8637033);    
    \draw[thick] (0, -3.8637033) -- (0, -2.8637033);    

    \filldraw[fill=white] (0 , 0) circle (2pt); 
    \filldraw[fill=white] (0 , -1) circle (2pt); 
    \filldraw[fill=white] (0.9659258 , -0.2588190) circle (2pt); 
    \filldraw[fill=white] (-0.9659258 , -0.2588190) circle (2pt); 
    \filldraw[fill=white] (1.6730326, -0.9659258) circle (2pt); 
    \filldraw[fill=white] (-1.6730326, -0.9659258) circle (2pt); 
    \filldraw[fill=white] (1.9318517, -1.9318517) circle (2pt); 
    \filldraw[fill=white] (-1.9318517, -1.9318517) circle (2pt) node[left, xshift = -0.1cm] {$n=12$};  
    \filldraw[fill=white] (1.6730326, -2.8977775) circle (2pt); 
    \filldraw[fill=white] (-1.6730326, -2.8977775) circle (2pt); 
    \filldraw[fill=white] (0.9659258, -3.6048843) circle (2pt); 
    \filldraw[fill=white] (-0.9659258, -3.6048843) circle (2pt); 
    \filldraw[fill=white] (0, -3.8637033) circle (2pt); 
    \filldraw[fill=white] (0, -2.8637033) circle (2pt); 

    \filldraw[fill=white] (-0.807007196, -2.3977775) circle (2pt); 
    \filldraw[fill=white] (0.807007196, -2.3977775) circle (2pt); 
    \filldraw[fill=white] (0.807007196, -1.4659258) circle (2pt); 
    \filldraw[fill=white] (-0.807007196, -1.4659258) circle (2pt); 
  \end{tikzpicture}
  \begin{tikzpicture}
    \draw[dashed] (1.060660172,  -2.992511872) -- (0, -0.43185165);
    \draw[dashed] (1.060660172,  -2.992511872) -- (-1.5, -1.9318517);
    \draw[dashed] (1.060660172,  -0.871191528) -- (-1.5, -1.9318517);
    \draw[dashed] (1.5, -1.9318517) -- (-1.5, -1.9318517);
    \draw[dashed] (1.5, -1.9318517) -- (-1.060660172,  -0.871191528);
    \draw[dashed] (1.060660172,  -2.992511872) -- (-1.060660172,  -0.871191528);
    \draw[dashed] (-1.060660172,  -2.992511872) -- (1.5, -1.9318517);
    \draw[dashed] (-1.060660172,  -2.992511872) -- (1.060660172,  -0.871191528);
    \draw[dashed] (-1.060660172,  -2.992511872) -- (0, -0.43185165);
    \draw[dashed] (0, -3.43185165) -- (1.060660172,  -0.871191528);
    \draw[dashed] (0, -3.43185165) -- (-1.060660172,  -0.871191528);

    \draw[thick] (0,0.7071068) -- (1.31950985, 0.353545201);
    \draw[thick] (0,0.7071068) -- (-1.31950985, 0.353545201);
    \draw[thick] (0,0.7071068) -- (0,0.137627575);

    \draw[thick] (1.31950985, 0.353545201) -- (2.285458101, -0.61234185);
    \draw[thick] (-1.31950985, 0.353545201) -- (-2.285458101, -0.61234185);    

    \draw[thick] (2.285458101, -0.61234185) -- (1.060660172,  -0.871191528)  node[below, xshift = 0.65cm, yshift = 0.2cm]  {6};
    \draw[thick] (-2.285458101, -0.61234185) -- (-1.060660172,  -0.871191528)  node[below, xshift = -0.65cm, yshift = 0.2cm]  {6};
    \draw[thick] (2.285458101, -3.25136155) -- (1.060660172,  -2.992511872)  node[below, xshift = 0.6cm, yshift = -0.053cm]  {6};
    \draw[thick] (-2.285458101, -3.25136155) -- (-1.060660172,  -2.992511872)  node[below, xshift = -0.6cm, yshift = -0.053cm]  {6};

    \draw[ultra thick] (1.060660172,  -0.871191528) -- (-1.060660172,  -0.871191528);
    \draw[ultra thick] (-1.5,  -1.9318517) -- (-2.6390197, -1.9318517);
    \draw[ultra thick] (2.6390197, -1.9318517) -- (1.5,  -1.9318517);
    \draw[ultra thick]  (-1.060660172,  -0.871191528) -- (-1.060660172,  -2.992511872);
    \draw[ultra thick] (1.060660172,  -2.992511872) -- (1.060660172,  -0.871191528);
    \draw[ultra thick] (1.060660172,  -2.992511872) -- (-1.060660172,  -2.992511872);
    \draw[ultra thick] (0, -0.43185165) -- (0, -3.43185165);

    \draw[thick] (2.285458101, -0.61234185) -- (2.6390197, -1.9318517);
    \draw[thick] (-2.285458101, -0.61234185) -- (-2.6390197, -1.9318517);    
    \draw[thick] (2.6390197, -1.9318517) -- (2.285458101, -3.25136155);
    \draw[thick] (-2.6390197, -1.9318517) -- (-2.285458101, -3.25136155);
    \draw[thick] (2.285458101, -3.25136155) -- (1.31950985, -4.217248501);
    \draw[thick] (-2.285458101, -3.25136155) -- (-1.31950985, -4.217248501);    
    \draw[thick] (1.31950985, -4.217248501) -- (0, -4.5708101);
    \draw[thick] (-1.31950985, -4.217248501) -- (0, -4.5708101);    
    \draw[thick] (0, -4.5708101) -- (0, -3.647777475);    

    \draw[thick, double distance = 2pt] (0,0.137627575) -- (0, -0.43185165);
    \draw[thick, double distance = 2pt] (0, -4) -- (0, -3.43185165);    
    \draw[thick, double distance = 2pt] (1.5, -1.9318517) -- (0, -0.43185165);
    \draw[thick, double distance = 2pt] (-1.5, -1.9318517) -- (0, -0.43185165);
    \draw[thick, double distance = 2pt] (1.5, -1.9318517) -- (0, -3.43185165);
    \draw[thick, double distance = 2pt] (-1.5, -1.9318517) -- (0, -3.43185165);

    \filldraw[fill=white] (0 , 0.7071068) circle (2pt); 
    \filldraw[fill=white] (1.31950985, 0.353545201) circle (2pt); 
    \filldraw[fill=white] (-1.31950985, 0.353545201) circle (2pt); 
    \filldraw[fill=white] (2.285458101, -0.61234185) circle (2pt); 
    \filldraw[fill=white] (-2.285458101, -0.61234185) circle (2pt); 
    \filldraw[fill=white] (2.6390197, -1.9318517) circle (2pt) node[right, xshift = 0.1cm] {$n=13$}; 
    \filldraw[fill=white] (-2.6390197, -1.9318517) circle (2pt); 
    \filldraw[fill=white] (2.285458101, -3.25136155) circle (2pt); 
    \filldraw[fill=white] (-2.285458101, -3.25136155) circle (2pt); 
    \filldraw[fill=white] (1.31950985, -4.217248501) circle (2pt); 
    \filldraw[fill=white] (-1.31950985, -4.217248501) circle (2pt); 
    \filldraw[fill=white] (0, -4) circle (2pt); 

    \filldraw[fill=white] (0,0.137627575) circle (2pt); 
    \filldraw[fill=white] (0, -4.5708101) circle (2pt); 
    \filldraw[fill=white] (1.5,  -1.9318517) circle (2pt); 
    \filldraw[fill=white] (-1.5,  -1.9318517) circle (2pt); 
    \filldraw[fill=white] (1.060660172,  -0.871191528) circle (2pt); 
    \filldraw[fill=white] (-1.060660172,  -0.871191528) circle (2pt); 
    \filldraw[fill=white] (-1.060660172,  -2.992511872) circle (2pt); 
    \filldraw[fill=white] (1.060660172,  -2.992511872) circle (2pt); 
    \filldraw[fill=white] (0, -3.43185165) circle (2pt); 
    \filldraw[fill=white] (0, -0.43185165) circle (2pt); 

  \end{tikzpicture}

\end{figure}
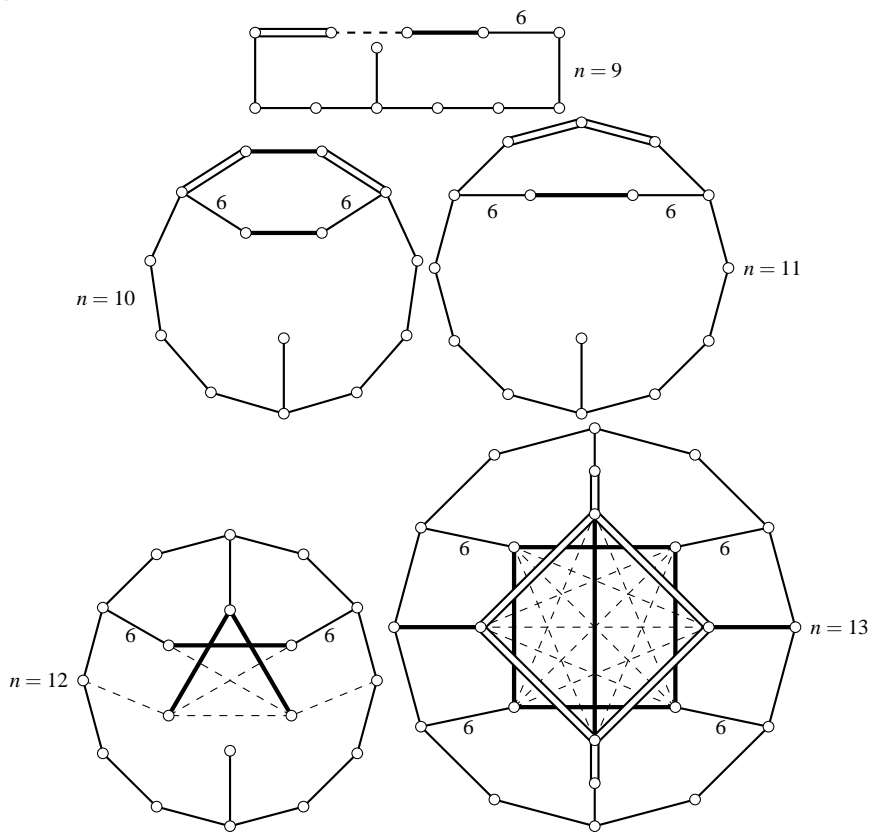

Based on the work of Chein \cite{chein69}, and Tumarkin \cite{tumarkin04}, it is possible to obtain a geometric description of some of these Coxeter polyhedra. 
\begin{corollary}
For $n=2, 3$, the geometric description of the polyhedra in Figure~\ref{one_to_eight} is a simplex. In two dimensions is it compact, and in three dimensions it is non-compact.

For $n=4, \ldots, 8$, the geometric description of the polyhedra in Figure~\ref{one_to_eight} is a pyramid over a product of two simplicies. These are non-compact polyhedra, and each have a single ideal vertex. In each of these cases, the hyperplane corresponding to the base of the pyramid is identified by a black vertex. This illustrates a result of Vinberg \cite{vinberg85} which states that parabolic subgraphs of rank $n-1$ correspond to ideal vertexes.
\end{corollary}

 In dimensions $9$-$13$, it is not possible to obtain a similar geometric description of the polyhedron for two reasons. The first is that the polyhedra are non-compact and have at least $n+4$ sides, a class which has not yet been explored in the literature (although the compact case has recieved some attention in \cite{tumarkinfelikson08}), and the second is that the Coxeter schemes contain a closed loop. 

It is not clear what geometric information about the polyhedron is encoded by a closed loop in the Coxeter scheme, beyond contributing to parabolic subgraphs.

Geometric information which can be recovered from the Coxeter scheme is an enumeration of the ideal vertices of the polyhedron. By \cite{vinberg85}, an ideal vertex is a parabolic subgraph of rank $n-1$. Ideal vertices correspond to the cusps of the quotient spaces. 

We can also describe the symmetry groups of the Coxeter polyhedra. By (\ref{sdproduct}), the group $\Theta_n$ is decomposed into a semi-direct product $\Gamma \rtimes H$. The symmetry group $Sym \; P$, of which $H$ is a subgroup, is naturally isomorphic to the symmetry group of the Coxeter scheme of $P$. In our case we always have, $H = Sym \; P$. This can be seen by inspection, in that any element $\eta \in Sym \; P$ swaps pairs of vectors $(e_i, e_j)$, and it can be seen that 
\[(\eta(e_i), \eta(e_j)) = (e_i, e_j)\]
so $Sym \; P$ preserves the lattice.

Therefore, by analysing the diagrams in Figure~\ref{ge_nine}, we can obtain the following corollary.

\begin{corollary}
For $n \le 9$, $Sym \; P$ is trivial, while for $10 \le n \le 12$, $Sym \; P$ is isomorphic to $\mathbb{Z}_2$.

For $n=9$ the polyhedron has two ideal vertices which are not symmetric to one another.

For $n=10$ the polyhedron has three ideal vertices, two of which are symmetrically placed.

For $n=11$ the polyhedron has five ideal vertices. These can be grouped into two pairs of symmetric vertices, and a single distinct vertex.

For $n=12$ the polyhedron has six ideal vertices. These can be grouped into two pairs of symmetric vertices, and two distinct vertices.

For $n=13$ the polyhedron has thirteen ideal vertices. These can be grouped into six pairs of symmetric vertices, and a single distinct vertex. The symmetry group of the Coxeter scheme in 13 dimensions is generated by two reflections which commute, and so is isomorphic to $\mathbb{Z}_2 \times \mathbb{Z}_2$.
\end{corollary}

\section{No examples in higher dimensions}\label{noexamplesinhigherdimensions}
The reflection groups presented so far are the only examples associated to this quadratic form. In this section, we prove that there are no higher dimensional examples, by showing that there is always a parabolic subgraph of insufficent rank, and it is impossible to produce a hyperplane which satisfies the crystallographic condition and completes the graph.

We now prove the second part of Theorem~\ref{theorem1}:
\begin{proposition}\label{noreflgroupsge14}
There are no discrete reflection groups associated to the quadratic form $-3x_0^2 + x_1^2 + \ldots + x_n^2$ in $n$-dimensions with $n \ge 14$ with finite covolume.
\end{proposition}

\begin{lemma}
For  $n \ge 14$, the first four  vectors generated by the algorithm are presented in Table~\ref{14d_vectors}.
\begin{table}[h!]
\centering
\caption{\label{14d_vectors}The first four vectors produced by Vinberg's algorithm in $14$ dimensions}
  \begin{tabular}{c|c|c|c|c}
    $i$ & $e_i$ & $(e,e)$ & $\frac{k_0^2}{(e,e)}$\\
    \hline
    $n+1$ & $v_0 + 3 v_1$ & $6$ & $0.167$\\
    $n+2$ & $v_0 + v_1 + v_2 + v_3 + v_4 + v_5$ & $2$ & $0.5$\\
    $n+3$ & $2(v_0 + v_1) + v_2 + \ldots + v_{11}$ & $2$ & $2$\\
    $n+4$ & $2v_0 + v_1 + \ldots + v_{14}$ & $2$ & $2$
  \end{tabular}
\end{table}
\end{lemma}

The proof of this lemma proceeds in the same way as the proof of Proposition~\ref{vector_prop}.

Consider the Coxeter scheme produced by taking the vectors in Table~\ref{14d_vectors} on top of the polyhedral angle. This Coxeter scheme (Figure~\ref{coxeter_scheme_14d}) describes a polyhedron which has infinite volume, and it can be used to prove Proposition~\ref{noreflgroupsge14}.

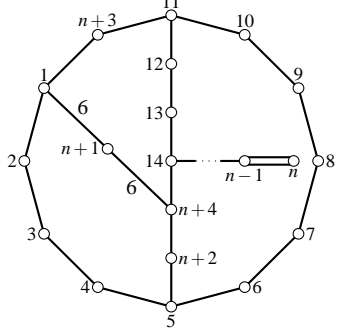
\begin{figure}
\centering
\caption{\label{coxeter_scheme_14d}The Coxeter scheme of the polyhedral angle along with the vectors in Table~\ref{14d_vectors}}
  \begin{tikzpicture}
    \draw[thick] (0,0) -- (0.9659258 , -0.2588190);
    \draw[thick] (0,0) -- (-0.9659258 , -0.2588190);
    \draw[thick] (0.9659258 , -0.2588190) -- (1.6730326, -0.9659258);
    \draw[thick] (-0.9659258 , -0.2588190) -- (-1.6730326, -0.9659258);    
    \draw[thick] (1.6730326, -0.9659258) -- (1.9318517, -1.9318517);
    \draw[thick] (-1.6730326, -0.9659258) -- (-1.9318517, -1.9318517);    
    \draw[thick] (1.9318517, -1.9318517) -- (1.6730326, -2.8977775);
    \draw[thick] (-1.9318517, -1.9318517) -- (-1.6730326, -2.8977775);    
    \draw[thick] (1.6730326, -2.8977775) -- (0.9659258, -3.6048843);
    \draw[thick] (-1.6730326, -2.8977775) -- (-0.9659258, -3.6048843);    
    \draw[thick] (0.9659258, -3.6048843) -- (0, -3.8637033);
    \draw[thick] (-0.9659258, -3.6048843) -- (0, -3.8637033);    
    \draw[thick] (0, -3.8637033) -- (0, -3.21975275);    
    \draw[thick] (0, -3.2197527) -- (0, -2.5758022);    
    \draw[thick] (0, -2.5758022) -- (0, -1.9318516);   
    \draw[thick] (0, -1.9318516) -- (0, -1.2879011);    

    \draw[thick] (0, -1.9318516) -- (0.321975275, -1.9318516);    
    \draw[dotted] (0.321975275, -1.9318516) -- (0.6439505, -1.9318516);    
    \draw[thick] (0.6439505, -1.9318516) -- (0.965925775, -1.9318516);    
    \draw[thick, double distance = 2pt] (0.965925775, -1.9318516) -- (1.609876275, -1.9318516);    

    \draw[thick] (0, -1.2879011) -- (0, -0.6439505);    
    \draw[thick] (0, -0.6439505) -- (0, 0);    

    \draw[thick] (-1.6730326, -0.9659258) -- (-0.8365163, -1.770864) node[xshift = -0.325cm, yshift = 0.55cm] {6};    
    \draw[thick] (-0.8365163, -1.770864) -- (0, -2.5758022) node[xshift = -0.525cm, yshift = 0.3cm] {6};    

    \filldraw[fill=white] (0 , 0) circle (2pt) node [above] {\scriptsize 11}; 
    \filldraw[fill=white] (0.9659258 , -0.2588190) circle (2pt) node [above] {\scriptsize 10}; 
    \filldraw[fill=white] (-0.9659258 , -0.2588190) circle (2pt) node [above] {\scriptsize $n+3$};
    \filldraw[fill=white] (1.6730326, -0.9659258) circle (2pt) node [above] {\scriptsize 9}; 
    \filldraw[fill=white] (-1.6730326, -0.9659258) circle (2pt) node [above] {\scriptsize 1}; 
    \filldraw[fill=white] (1.9318517, -1.9318517) circle (2pt) node [right] {\scriptsize 8}; 
    \filldraw[fill=white] (-1.9318517, -1.9318517) circle (2pt) node [left] {\scriptsize 2}; 
    \filldraw[fill=white] (1.6730326, -2.8977775) circle (2pt) node [right] {\scriptsize 7}; 
    \filldraw[fill=white] (-1.6730326, -2.8977775) circle (2pt) node [left] {\scriptsize 3}; 
    \filldraw[fill=white] (0.9659258, -3.6048843) circle (2pt) node [right] {\scriptsize 6}; 
    \filldraw[fill=white] (-0.9659258, -3.6048843) circle (2pt) node [left] {\scriptsize 4}; 
    \filldraw[fill=white] (0, -3.8637033) circle (2pt) node [below] {\scriptsize 5}; 
    \filldraw[fill=white] (0, -3.2197527) circle (2pt) node [right] {\scriptsize $n+2$}; 
    \filldraw[fill=white] (0, -2.5758022) circle (2pt) node [right] {\scriptsize $n+4$}; 
    \filldraw[fill=white] (0, -1.9318516) circle (2pt) node [left] {\scriptsize 14};  
    \filldraw[fill=white] (0, -1.2879011) circle (2pt) node [left] {\scriptsize 13}; 
    \filldraw[fill=white] (0, -0.6439505) circle (2pt) node [left] {\scriptsize 12};  

    \filldraw[fill=white] (0.965925775, -1.9318516) circle (2pt) node [below] {\scriptsize $n-1$}; 
    \filldraw[fill=white] (1.609876275, -1.9318516) circle (2pt) node [below] {\scriptsize $n$}; 

    \filldraw[fill=white] (-0.8365163, -1.770864) circle (2pt) node [left] {\scriptsize $n+1$};  
  \end{tikzpicture}
\end{figure}

A parabolic subgraph of this diagram is a pair of copies of $\tilde{E_6}$ (vertexes 1, 9, 10, 11, 12, 13, $n+3$; and 3, 4, 5, 6, 7, $n+2$, $n+4$), which will be denoted $\Gamma_p$. $\Gamma_p$ has rank 12, and in order for the polyhedron to have finite volume, it must be extended to have rank $n-1$.

\begin{figure}
\centering
\caption{\label{isolateE6}Isolating $\Gamma_p$ as a parabolic subgraph of Figure~\ref{coxeter_scheme_14d}.}
  \begin{tikzpicture}
    \draw[thick] (0,0) -- (0.9659258 , -0.2588190);
    \draw[thick] (0,0) -- (-0.9659258 , -0.2588190);
    \draw[thick] (0.9659258 , -0.2588190) -- (1.6730326, -0.9659258);
    \draw[thick] (-0.9659258 , -0.2588190) -- (-1.6730326, -0.9659258);    

    \draw[thick] (1.6730326, -2.8977775) -- (0.9659258, -3.6048843);
    \draw[thick] (-1.6730326, -2.8977775) -- (-0.9659258, -3.6048843);    
    \draw[thick] (0.9659258, -3.6048843) -- (0, -3.8637033);
    \draw[thick] (-0.9659258, -3.6048843) -- (0, -3.8637033);    
    \draw[thick] (0, -3.8637033) -- (0, -3.21975275);    
    \draw[thick] (0, -3.2197527) -- (0, -2.5758022);    

    \draw[thick] (0, -1.9318516) -- (0.321975275, -1.9318516);    
    \draw[dotted] (0.321975275, -1.9318516) -- (0.6439505, -1.9318516);    
    \draw[thick] (0.6439505, -1.9318516) -- (0.965925775, -1.9318516);    
    \draw[thick, double distance = 2pt] (0.965925775, -1.9318516) -- (1.609876275, -1.9318516);    

    \draw[thick] (0, -1.2879011) -- (0, -0.6439505);    
    \draw[thick] (0, -0.6439505) -- (0, 0);    

    \filldraw[fill=white] (0 , 0) circle (2pt) node [above] {\scriptsize 11}; 
    \filldraw[fill=white] (0.9659258 , -0.2588190) circle (2pt) node [above] {\scriptsize 10}; 
    \filldraw[fill=white] (-0.9659258 , -0.2588190) circle (2pt) node [above] {\scriptsize $n+3$};
    \filldraw[fill=white] (1.6730326, -0.9659258) circle (2pt) node [above] {\scriptsize 9}; 
    \filldraw[fill=white] (-1.6730326, -0.9659258) circle (2pt) node [above] {\scriptsize 1}; 
    \filldraw[fill=white] (1.6730326, -2.8977775) circle (2pt) node [right] {\scriptsize 7}; 
    \filldraw[fill=white] (-1.6730326, -2.8977775) circle (2pt) node [left] {\scriptsize 3}; 
    \filldraw[fill=white] (0.9659258, -3.6048843) circle (2pt) node [right] {\scriptsize 6}; 
    \filldraw[fill=white] (-0.9659258, -3.6048843) circle (2pt) node [left] {\scriptsize 4}; 
    \filldraw[fill=white] (0, -3.8637033) circle (2pt) node [below] {\scriptsize 5}; 
    \filldraw[fill=white] (0, -3.2197527) circle (2pt) node [right] {\scriptsize $n+2$}; 
    \filldraw[fill=white] (0, -2.5758022) circle (2pt) node [right] {\scriptsize $n+4$}; 
    \filldraw[fill=white] (0, -1.9318516) circle (2pt) node [left] {\scriptsize 15};  
    \filldraw[fill=white] (0, -1.2879011) circle (2pt) node [left] {\scriptsize 13}; 
    \filldraw[fill=white] (0, -0.6439505) circle (2pt) node [left] {\scriptsize 12};  

    \filldraw[fill=white] (0.965925775, -1.9318516) circle (2pt) node [below] {\scriptsize $n-1$}; 
    \filldraw[fill=white] (1.609876275, -1.9318516) circle (2pt) node [below] {\scriptsize $n$}; 
  \end{tikzpicture}
\end{figure}
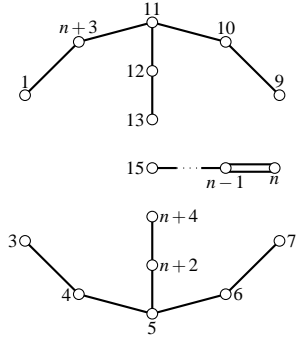

Deleting the vertexes which are connected to $\Gamma_p$ demonstrates that there are three connected components, shown in Figure~\ref{isolateE6} (note that when $n=14$ there are only two connected components). The third component is a copy of the elliptic graph $B_{n-14}$ (note that in 15 dimensions the third component is a copy of the elliptic graph $A_1$). Therefore new vertexes must be added to make another parabolic subgraph (possibly containing the elliptic graph) of rank $n-13$. These new vertexes must not have edges to $\Gamma_p$, otherwise they will immediately be deleted while isolating the parabolic subgraph.

Therefore the inner product of the new vectors with the vectors comprising $\Gamma_p$ must be zero.

\begin{proof}{(Proposition~\ref{noreflgroupsge14})}
The new vector $e$ will be written as
\[e = \sum_{i=0}^n k_iv_i.\]

All of the vectors numbered $1$-$(n-1)$ are of the form $-v_i + v_{i+1}$ and as $e$ must have zero inner product with the vertexes of the $\Gamma_p$ labelled 1, 3, 4, 5, 6, 7, 9, 10, 11, 12, 13,  we will define
\begin{eqnarray*}
k_1 = k_2  &=: m,\\
k_3 = k_4 = k_5 = k_6 = k_7 = k_8 &=: p,\\
k_9 = k_{10} = k_{11} = k_{12} = k_{13} = k_{14} &=: q.
\end{eqnarray*}

Consider the vertex labelled $(n+2)$.  If $e$ has zero inner product with the $v_0 + v_1 + v_2 + v_3 + v_4 + v_5$ it implies that
\[3 k_0 = 2m + 3p.\]

Now consider the vertex labelled $(n+3)$. Similarly we get
\[6k_0 = 3m + 6p + 3q.\]

Finally, consider the vertex labelled $(n+4)$. We get
\[6k_0 = 2m + 6p + 6q.\]

These last two expressions can be subtracted from one another to show that
\[3q = m,\]
which implies that
\[k_0 = 2q + p,\]
hence  we can write $e$ as
\begin{equation}\label{vector}
e = (2q + p) v_0 + 3q(v_1 + v_2) + p(v_3 + v_4 + \ldots + v_8) + q(v_9 + v_{10} + \ldots + v_{14}) + \sum_{i=15}^n k_i v_i.
\end{equation}

This vector has (squared) length
\begin{equation}\label{squared_length}
|e|^2 = 3(p-2q)^2 + \sum_{i=15}^n k_i^2.
\end{equation}

By the crystallographic condition, this quantity must be 1, 2, 3, or 6, and if it is equal to 3 or 6 then all of the coefficients (including $p$ and $q$) must be divisible by 3. Therefore (\ref{squared_length}) is given by
\[|e|^2 = 27(p^\prime-2q^\prime)^2 + 9\sum_{i=15}^n {k_i^\prime}^2,\]
where $p = 3p^\prime$, $q = 3q^\prime$, and $k_i = 3{k_i^\prime}$. This cannot equal 3 or 6.

By (\ref{vector}), $p \ge q > 0$, and $q \ge k_{15} \ge \ldots \ge k_{n-1} \ge k_n \ge 0$, so in 14 dimensions, (\ref{squared_length}) cannot equal 1 or 2. Therefore, in 14 dimensions, the algorithm does not produce a polyhedron of finite volume. In 15 dimensions the vector can be of length 1 if $p = 2q$ and $k_{15} = 1$, and in higher dimensions the vector can be of length 2 if in addition, $k_{16}$ is also 1. For fixed $k_0$ (as in this case) the longer vector represents a closer mirror, and so in dimension $\ge 16$ we must consider $e$ to have length 2.

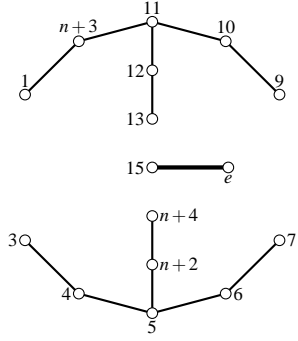
\begin{figure}
\centering
\caption{\label{15_dim}Including the vector $e$ in 15 dimensions}
  \begin{tikzpicture}
    \draw[thick] (0,0) -- (0.9659258 , -0.2588190);
    \draw[thick] (0,0) -- (-0.9659258 , -0.2588190);
    \draw[thick] (0.9659258 , -0.2588190) -- (1.6730326, -0.9659258);
    \draw[thick] (-0.9659258 , -0.2588190) -- (-1.6730326, -0.9659258);    

    \draw[thick] (1.6730326, -2.8977775) -- (0.9659258, -3.6048843);
    \draw[thick] (-1.6730326, -2.8977775) -- (-0.9659258, -3.6048843);    
    \draw[thick] (0.9659258, -3.6048843) -- (0, -3.8637033);
    \draw[thick] (-0.9659258, -3.6048843) -- (0, -3.8637033);    
    \draw[thick] (0, -3.8637033) -- (0, -3.21975275);    
    \draw[thick] (0, -3.2197527) -- (0, -2.5758022);    

    \draw[ultra thick] (0, -1.9318516) -- (1, -1.9318516);    

    \draw[thick] (0, -1.2879011) -- (0, -0.6439505);    
    \draw[thick] (0, -0.6439505) -- (0, 0);    

    \filldraw[fill=white] (0 , 0) circle (2pt) node [above] {\scriptsize 11}; 
    \filldraw[fill=white] (0.9659258 , -0.2588190) circle (2pt) node [above] {\scriptsize 10}; 
    \filldraw[fill=white] (-0.9659258 , -0.2588190) circle (2pt) node [above] {\scriptsize $n+3$};
    \filldraw[fill=white] (1.6730326, -0.9659258) circle (2pt) node [above] {\scriptsize 9}; 
    \filldraw[fill=white] (-1.6730326, -0.9659258) circle (2pt) node [above] {\scriptsize 1}; 
    \filldraw[fill=white] (1.6730326, -2.8977775) circle (2pt) node [right] {\scriptsize 7}; 
    \filldraw[fill=white] (-1.6730326, -2.8977775) circle (2pt) node [left] {\scriptsize 3}; 
    \filldraw[fill=white] (0.9659258, -3.6048843) circle (2pt) node [right] {\scriptsize 6}; 
    \filldraw[fill=white] (-0.9659258, -3.6048843) circle (2pt) node [left] {\scriptsize 4}; 
    \filldraw[fill=white] (0, -3.8637033) circle (2pt) node [below] {\scriptsize 5}; 
    \filldraw[fill=white] (0, -3.2197527) circle (2pt) node [right] {\scriptsize $n+2$}; 
    \filldraw[fill=white] (0, -2.5758022) circle (2pt) node [right] {\scriptsize $n+4$}; 
    \filldraw[fill=white] (0, -1.9318516) circle (2pt) node [left] {\scriptsize 15};  
    \filldraw[fill=white] (0, -1.2879011) circle (2pt) node [left] {\scriptsize 13}; 
    \filldraw[fill=white] (0, -0.6439505) circle (2pt) node [left] {\scriptsize 12};  
    \filldraw[fill=white] (1, -1.9318516) circle (2pt) node [below] {\scriptsize $e$};  

  \end{tikzpicture}
\end{figure}

\begin{figure}
\centering
\caption{\label{16_dim}Including the vector $e$ in 16 dimensions}
  \begin{tikzpicture}
    \draw[thick] (0,0) -- (0.9659258 , -0.2588190);
    \draw[thick] (0,0) -- (-0.9659258 , -0.2588190);
    \draw[thick] (0.9659258 , -0.2588190) -- (1.6730326, -0.9659258);
    \draw[thick] (-0.9659258 , -0.2588190) -- (-1.6730326, -0.9659258);    

    \draw[thick] (1.6730326, -2.8977775) -- (0.9659258, -3.6048843);
    \draw[thick] (-1.6730326, -2.8977775) -- (-0.9659258, -3.6048843);    
    \draw[thick] (0.9659258, -3.6048843) -- (0, -3.8637033);
    \draw[thick] (-0.9659258, -3.6048843) -- (0, -3.8637033);    
    \draw[thick] (0, -3.8637033) -- (0, -3.21975275);    
    \draw[thick] (0, -3.2197527) -- (0, -2.5758022);    

    \draw[thick, double distance = 2pt] (0, -1.9318516) -- (1, -1.9318516);    
    \draw[thick, double distance = 2pt] (1, -1.9318516) -- (2, -1.9318516);    

    \draw[thick] (0, -1.2879011) -- (0, -0.6439505);    
    \draw[thick] (0, -0.6439505) -- (0, 0);    

    \filldraw[fill=white] (0 , 0) circle (2pt) node [above] {\scriptsize 11}; 
    \filldraw[fill=white] (0.9659258 , -0.2588190) circle (2pt) node [above] {\scriptsize 10}; 
    \filldraw[fill=white] (-0.9659258 , -0.2588190) circle (2pt) node [above] {\scriptsize $n+3$};
    \filldraw[fill=white] (1.6730326, -0.9659258) circle (2pt) node [above] {\scriptsize 9}; 
    \filldraw[fill=white] (-1.6730326, -0.9659258) circle (2pt) node [above] {\scriptsize 1}; 
    \filldraw[fill=white] (1.6730326, -2.8977775) circle (2pt) node [right] {\scriptsize 7}; 
    \filldraw[fill=white] (-1.6730326, -2.8977775) circle (2pt) node [left] {\scriptsize 3}; 
    \filldraw[fill=white] (0.9659258, -3.6048843) circle (2pt) node [right] {\scriptsize 6}; 
    \filldraw[fill=white] (-0.9659258, -3.6048843) circle (2pt) node [left] {\scriptsize 4}; 
    \filldraw[fill=white] (0, -3.8637033) circle (2pt) node [below] {\scriptsize 5}; 
    \filldraw[fill=white] (0, -3.2197527) circle (2pt) node [right] {\scriptsize $n+2$}; 
    \filldraw[fill=white] (0, -2.5758022) circle (2pt) node [right] {\scriptsize $n+4$}; 
    \filldraw[fill=white] (0, -1.9318516) circle (2pt) node [left] {\scriptsize 15};  
    \filldraw[fill=white] (0, -1.2879011) circle (2pt) node [left] {\scriptsize 13}; 
    \filldraw[fill=white] (0, -0.6439505) circle (2pt) node [left] {\scriptsize 12};  

    \filldraw[fill=white] (1, -1.9318516) circle (2pt) node [below] {\scriptsize $16$}; 
    \filldraw[fill=white] (2, -1.9318516) circle (2pt) node [below] {\scriptsize $e$}; 
  \end{tikzpicture}
\end{figure}
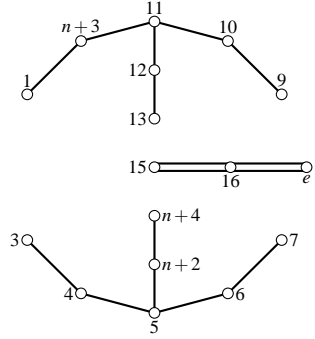

\begin{figure}
\centering
\caption{\label{ge17_dim}Including the vector $e$ in $\ge 17$ dimensions}
  \begin{tikzpicture}
    \draw[thick] (0,0) -- (0.9659258 , -0.2588190);
    \draw[thick] (0,0) -- (-0.9659258 , -0.2588190);
    \draw[thick] (0.9659258 , -0.2588190) -- (1.6730326, -0.9659258);
    \draw[thick] (-0.9659258 , -0.2588190) -- (-1.6730326, -0.9659258);    

    \draw[thick] (1.6730326, -2.8977775) -- (0.9659258, -3.6048843);
    \draw[thick] (-1.6730326, -2.8977775) -- (-0.9659258, -3.6048843);    
    \draw[thick] (0.9659258, -3.6048843) -- (0, -3.8637033);
    \draw[thick] (-0.9659258, -3.6048843) -- (0, -3.8637033);    
    \draw[thick] (0, -3.8637033) -- (0, -3.21975275);    
    \draw[thick] (0, -3.2197527) -- (0, -2.5758022);    

    \draw[thick] (0, -1.9318516) -- (1, -1.9318516);    
    \draw[thick] (1, -1.9318516) -- (1.333, -1.9318516);    
    \draw[thick] (1, -1.9318516) -- (1, -0.9318516);    
    \draw[thick] (1, -1.9318516) -- (1.333, -1.9318516);    
    \draw[dotted] (1.333, -1.9318516) -- (1.666, -1.9318516);    
    \draw[thick] (1.666, -1.9318516) -- (2, -1.9318516);    
    \draw[thick, double distance = 2pt] (2, -1.9318516) -- (3, -1.9318516);    

    \draw[thick] (0, -1.2879011) -- (0, -0.6439505);    
    \draw[thick] (0, -0.6439505) -- (0, 0);    

    \filldraw[fill=white] (0 , 0) circle (2pt) node [above] {\scriptsize 11}; 
    \filldraw[fill=white] (0.9659258 , -0.2588190) circle (2pt) node [above] {\scriptsize 10}; 
    \filldraw[fill=white] (-0.9659258 , -0.2588190) circle (2pt) node [above] {\scriptsize $n+3$};
    \filldraw[fill=white] (1.6730326, -0.9659258) circle (2pt) node [above] {\scriptsize 9}; 
    \filldraw[fill=white] (-1.6730326, -0.9659258) circle (2pt) node [above] {\scriptsize 1}; 
    \filldraw[fill=white] (1.6730326, -2.8977775) circle (2pt) node [right] {\scriptsize 7}; 
    \filldraw[fill=white] (-1.6730326, -2.8977775) circle (2pt) node [left] {\scriptsize 3}; 
    \filldraw[fill=white] (0.9659258, -3.6048843) circle (2pt) node [right] {\scriptsize 6}; 
    \filldraw[fill=white] (-0.9659258, -3.6048843) circle (2pt) node [left] {\scriptsize 4}; 
    \filldraw[fill=white] (0, -3.8637033) circle (2pt) node [below] {\scriptsize 5}; 
    \filldraw[fill=white] (0, -3.2197527) circle (2pt) node [right] {\scriptsize $n+2$}; 
    \filldraw[fill=white] (0, -2.5758022) circle (2pt) node [right] {\scriptsize $n+4$}; 
    \filldraw[fill=white] (0, -1.9318516) circle (2pt) node [left] {\scriptsize 15};  
    \filldraw[fill=white] (0, -1.2879011) circle (2pt) node [left] {\scriptsize 13}; 
    \filldraw[fill=white] (0, -0.6439505) circle (2pt) node [left] {\scriptsize 12};  

    \filldraw[fill=white] (1, -1.9318516) circle (2pt) node [below] {\scriptsize $16$}; 
    \filldraw[fill=white] (1, -0.9318516) circle (2pt) node [right] {\scriptsize $e$}; 
    \filldraw[fill=white] (2, -1.9318516) circle (2pt) node [below] {\scriptsize $n-1$};
    \filldraw[fill=white] (3, -1.9318516) circle (2pt) node [below] {\scriptsize $n$}; 
  \end{tikzpicture}
\end{figure}

As can be seen in Figure~\ref{15_dim} (respectively Figure~\ref{16_dim}; Figure~\ref{ge17_dim}), in 15 (respectively 16; $\ge 17$) dimensions, $e$ forms a copy of $\tilde{A_1}$ (respectively $\tilde{C_2}$; $\tilde{B_{n-14}}$) with the vertex(es) labelled 15 (respectively 15 and 16; 15, 16, $\ldots, n$). Along with the copies of $\tilde{E_6}$, this parabolic subgraph has rank 13 (respectively 14; $n-2$), which is insufficent to produce a finite volume polyhedron. New vectors still have to satisfy all of the above constraints, and are therefore of the form (\ref{vector}), but they must now also have zero inner product with $e_{15} = -v_{15}$ (respectively $e_{15} = -v_{15} + v_{16}$ and $e_{16} = -v_{16}$; $e_i = -v_i + v_{i+1}$, $15 \le i \le n-1$  and $e_n = -v_n$), so $k_{15}$ must be zero (respectively $k_{15}$ and $k_{16}$; $k_i$, $i \ge 15$). Therefore the vector must satisfy
\[|e|^2 = 3(p-2q)^2 = 1 \text{ or } 2\]
which, as we have already seen, is impossible. Therefore, in $\ge 14$ dimensions, the algorithm does not terminate.

There is no possibility of enlarging $\Gamma_p$ into a parabolic graph of rank $n-1$, and the polyhedron will have infinite volume for $n\ge 14$, so there are no further hyperbolic reflective lattices associated to this quadratic form. This completes the proof of Theorem~\ref{theorem1}.
\end{proof}

\begin{acknowledgements}
The author would like to thank Mikhail Belolipetsky for many valuable discussions, and {\`E.} Vinberg for his helpful comment on a previous version of this article.
\end{acknowledgements}

\bibliographystyle{spmpsci}      
\bibliography{references}   

\end{document}